\documentclass{amsart}

\usepackage{amsmath,amssymb,amsthm,a4wide}
\usepackage{url}
\usepackage{bbm}
\usepackage{graphicx,tikz}
\usetikzlibrary{shapes,arrows}
\usetikzlibrary{calc, positioning, shapes, through, intersections}
\newtheorem{theorem}{Theorem}
\newtheorem*{theorem*}{Theorem}
\newtheorem*{lemma*}{Lemma}

\newtheorem{corollary}[theorem]{Corollary}

\theoremstyle{definition}

\theoremstyle{remark}

\begin{document}

\title[]{The Wasserstein Distances Between Pushed-Forward Measures with Applications to Uncertainty Quantification}
\keywords{Wasserstein, Uncertainty-Quantification, Approximation.}
\subjclass[2010]{28A10, 60A10, 65D99.} 

\author[]{Amir Sagiv}
\address{Department of Applied Mathematics, Tel Aviv University, Tel Aviv 6997801, Israel}
\email{asagiv88@gmail.com}
%\thanks{This work is supported by the NSF (DMS-1763179) and the Alfred P. Sloan Foundation.}

\begin{abstract} 
In the study of dynamical and physical systems, the input parameters are often uncertain or randomly distributed according to a measure $\varrho$. The system's response $f$ pushes forward $\varrho$ to a new measure $f_* \varrho$ which we would like to study. However, we might not have access to $f$, but to its approximation $g$. This problem is common in the use of surrogate models for numerical uncertainty quantification (UQ). We thus arrive at a fundamental question -- if~$f$ and $g$ are close in an $L^q$ space, does the measure $g_* \varrho$ approximate $f_* \varrho$ well, and in what sense? Previously, it was demonstrated that the answer to this question might be negative when posed in terms of the $L^p$ distance between probability density functions (PDF). Instead, we show in this paper that the Wasserstein metric is the proper framework for this question. For domains in $\mathbb{R}^d$, we bound the Wasserstein distance~$W_p (f_* \varrho , g_* \varrho) $ from above by $\|f-g\|_{q}$. Furthermore, we prove lower bounds for for the cases where~$p=1$ and $p=2$ (for $d=1$) in terms of moments approximation. From a numerical analysis standpoint, since the Wasserstein distance is related to the cumulative distribution function~(CDF), we show that the latter is well approximated by methods such as spline interpolation and generalized polynomial chaos (gPC).
\end{abstract}

\maketitle

\section{Introduction}

\subsection{Problem formulation} Suppose a domain $\Omega \subseteq \mathbb{R}^d$ is equipped with a Borel probability measure~$\varrho$ and that a function $f:\Omega \to \mathbb{R}$ pushes forward $\varrho$ to a new measure $\mu:\,= f_* \varrho$, i.e., $f_* \mu (B) =~\varrho (f^{-1} (B))$ for every Borel set $B\subseteq \mathbb{R}$. We wish to characterize~$\mu$, but only have access to a function $g$ which approximates $f$. If $\|f-g\|_{L^q(\Omega, \varrho)}$ is small, does $\nu :\,=g_* \varrho$ approximate $\mu$ well, and if so in what sense?

% Define block styles
\tikzstyle{decision} = [diamond, draw, fill=blue!20, 
    text width=4.5em, text badly centered, node distance=3cm, inner sep=0pt]
\tikzstyle{block} = [rectangle, draw,  
    text width=3.5cm, text centered, rounded corners, minimum height=4em, minimum width=4cm]
\tikzstyle{line} = [draw, -latex']
\tikzstyle{cloud} = [draw, ellipse,fill=red!20, node distance=3cm,
    minimum height=2em]
    \begin{figure}[h!]
    \centering
\begin{tikzpicture}[node distance = 2cm, auto]
    % Place nodes
    \node [block] (f) {original $f$};
    \node [block, right= 4cm of f] (mu) {measure of interest, inaccessible $\mu :\,=f_* \varrho$};      
    \node [block, below= 2cm of f] (g) {surrogate $g$};      
        \node [block, right= 4cm of g] (nu) {accessible $\nu :\,=g_* \varrho$};      
	\path [line] (f) -- node {pushforward of $\varrho$} (mu);
		\path [line] (g) -- node {pushforward of $\varrho$} (nu);
		\path[line] (f)--node {approximation} (g);
		\path[line, dashed, blue] (nu)--node [blue, right]{approximation?}(mu);
	
\end{tikzpicture}
\caption{The schematic structure of the problem. If $\|f-g\|_p$ is small, how close are $\mu$ and $\nu$? In other words, is the dashed arrow ``justified"?}
\label{fig:problem}
\end{figure}
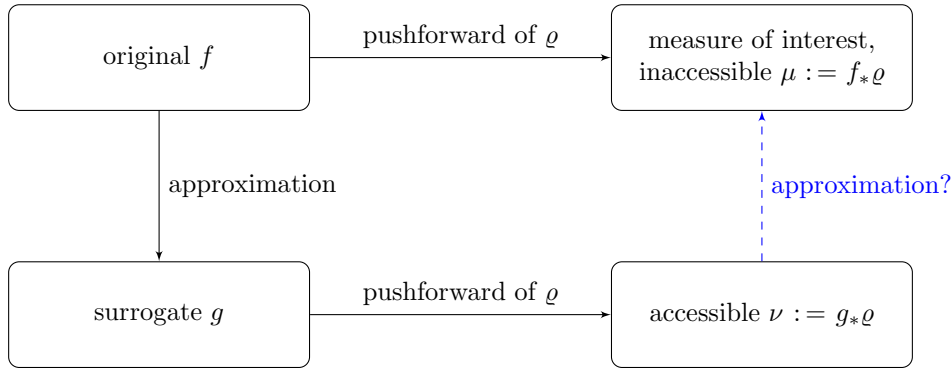
    
\subsection{Motivation}\label{sec:ode} To motivate this rather abstract question, consider the following toy example: a harmonic oscillator is described by the ordinary differential equation (ODE) $y''(t)+y=0$ with $y(0)=0$ and $y'(0)=v$. Suppose we are interested in $f(v)=~y^2(\pi/2; v)$. By solving this ODE, we know that $f(v) = [v\sin (\pi/2)]^2 = v^2$. In many other cases, however, we do not have direct access to $f$, but only to its approximation~$g$. This could happen for various reasons -- it may be that we can only compute $f(\alpha)$ numerically, or that we approximate $f$ using an asymptotic method. Following on the harmonic oscillator example, suppose we know $f(v)$ only at four given points $v_1$, $v_2$, $v_3$, and $v_4$. For any other value of~$v$, we approximate $f(v)$ by $g(v)$, which linearly interpolates the adjacent values of $f$, see Fig.\ \ref{fig:ode}(a).

The parameters and inputs of physical systems are often noisy or uncertain. We thus assume in the harmonic oscillator example that the initial speed $v$ is drawn uniformly at random from $[1,2]$. In these settings, $f(v)$ is random, and we are interested in the {\em distribution} of $f(v)$ over many experiments. Even though $f$ and $g$ look similar in Fig.\ \ref{fig:ode}(a), the probability density functions~(PDF) of $\mu=f_* \varrho$ and $\nu=g_*\varrho$, denoted by $p_{\mu}$ and $p_{\nu}$ respectively, are quite different, see Fig.\ \ref{fig:ode}(b). We would therefore like to have guarantees that $\nu$ approximates the original measure of interest~$\mu$~well.

\begin{figure}[h!]
\centering
{\includegraphics[scale=.5]{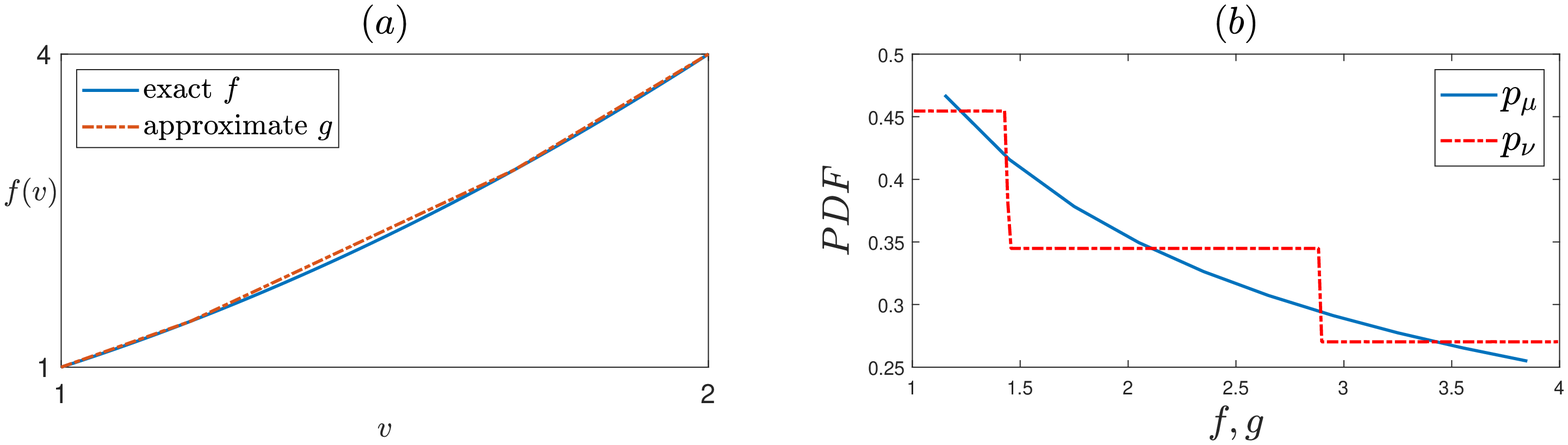}}
\caption{Solutions of $y''(t)+y=0$ with $y'(0)=v$ and $y(0)=0$. (a) $f(v)=y^2(t=\pi /2; v)$ (solid) and $g(v)$, its piecewise linear interpolant based on four exact samples (dash-dots). (b) The PDFs of $\mu = f_* \varrho$ (solid) and $\nu = g_* \varrho$~(dash-dots), where $\varrho$ is the uniform distribution on $[1,2]$.}
\label{fig:ode}
\end{figure}

It might seem obvious that the distance between $f$ and $g$ controls the distance between $\mu$ and~$\nu$. This hypothesis fails, however, when one estimates this distance using the PDFs $p_{\mu}$ and~$p_{\nu}$. For example, let $f(\alpha) =\alpha$ and $g(\alpha) = \alpha + \delta\sin ((10\delta)^{-1}\alpha)$, where $1\gg \delta > 0$. Since $\|f-g\|_{\infty} = \delta $, the two functions are seemingly indistinguishable from each other, see Fig.~\ref{fig:func_pdf}(a). Consider the case where~$\varrho$ is the Lebesgue measure on $[0,1]$	. Then, since both functions are monotonic, $p_{\mu}(y)=~1/f'(f^{-1}(y)) =~1$ and $p_{\nu}(y) = 1/g'(g^{-1}(y))$, see \cite{sagiv2018spline} for details. Hence, $p_{\nu}$ is onto~$[1.1^{-1},0.9^{-1}]\approx [0.91,1.11]$ and so $\|p_{\mu}-p_{\nu} \|_{\infty}> 0.1$, irrespectively of $\delta$, see Fig.~\ref{fig:func_pdf}(b). The lack of apparent correspondence between $\|f-g\|_q$ and $\|p_{\mu} - p_{\nu} \|_p$ for {\em any} pair of integers~$p$ and $q$ suggests that the PDFs are not a well-suited metric for the problem depicted in Fig.\ \ref{fig:problem}. {\em Instead, in this paper we propose the Wasserstein distance as the proper framework to measure the distance between $\mu$ and $\nu$.}

\begin{figure}[h!]
\centering
{\includegraphics[scale=.5]{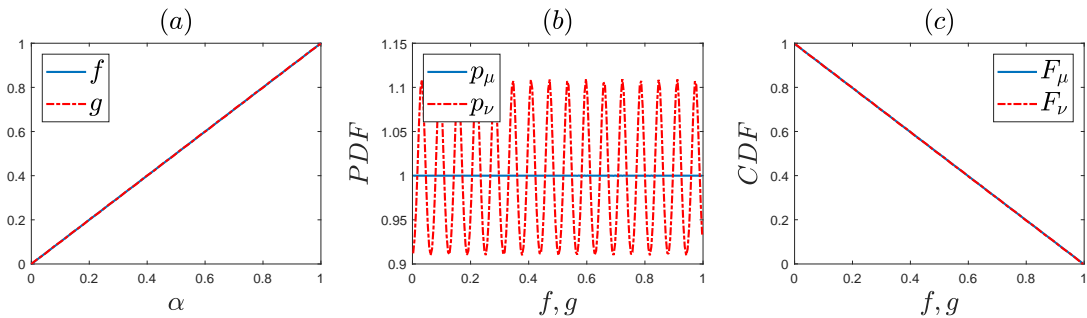}}
\caption{(a) $f(\alpha)=\alpha$ (solid) and $g(\alpha) = \alpha +10^{-3}\sin(100\alpha)$ (dash-dots). The two lines are indistinguishable. (b) The PDFs of $\mu =f_* \varrho$ and $\nu = g_* \varrho$, where $\varrho$ is the Lebesgue measure on $\Omega = [0,1]$. (c) The CDFs of the same measures. The two lines are indistinguishable.}
\label{fig:func_pdf}
\end{figure}

%Quantitatively, $\|f-g\|_1 =6.4\cdot 10^{-4}$ whereas $\|p_{\mu} - p_{\nu} \|_1 = 6.2\cdot 10^{-2}$, and $\|f-g\|_{\infty} = 10^{-3}$ whereas $\|p_{\mu} - p_{\nu}\|_{\infty} =0.11$.

\subsection{Relevant literature}

The harmonic oscillator example in Sec.\ \ref{sec:ode} serves as a toy example for a broad class of problems. While the ODE $y''(t) +y=0$ can be solved explicitly, many other differential equations do not admit such closed-form solutions. Instead, we only have an approximation for the quantities of interest at our disposal. Indeed, the general settings presented above have spurred numerous papers in a field of computational science known as Uncertainty-Quantification (UQ), see e.g.,~\cite{sagiv2018spline, halder2018adaptive,
sudret2000stochastic,
xiu2010numerical, xiu2005colocation, xiu2002galerkin}.

Perhaps surprisingly, the full approximation of $\mu$ (rather than its moments alone) in these particular settings received little theoretical attention in the literature, even though it is of practical importance in diverse fields such as ocean waves \cite{ ablowitz2015interacting},  computational fluid dynamics~\cite{ chen2005uncertainty},   hydrology~\cite{ colombo2018basins}, aeronautics \cite{ zabaras2007sparse}, biochemistry \cite{le2010asynchronous}, and nonlinear optics \cite{gauri2019polar, best2017paper}. Even though $\|f-g\|_q$ does not control $\|p_{\mu} - p_{\nu} \|_p$ in general (see e.g., Fig.\ \ref{fig:func_pdf}), a previous result by Ditkowski, Fibich, and the author gives sufficient conditions for PDF approximation:
\begin{theorem}[Ditkowski, Fibich, and Sagiv \cite{sagiv2018spline}]\label{thm:l1conv_pdf}
Let $f\in C^{2}([0,1]^d)$ and let $g_h\in~C^{2} ([0,1]^d)$ be an interpolant of $f$ on a tensor grid of maximal spacing $h>0$ such that $$\|f-g_h\|_{\infty} \, , \|\nabla f-~\nabla g_h \|_{\infty} \leq Kh^{\tau} \, ,$$ where $K=K(f,d)$ and $\tau >0$ is fixed. Then $$\|p_{\mu} - p_{\nu} \|_{L^p} \leq \tilde{K}h^{\tau} \, , $$ for every $1\leq p <\infty$, with a constant $\tilde{K}=\tilde{K}(f,d,q)$.
\end{theorem}
The conditions on $g$ are motivated by spline interpolation method, see Sec.\ \ref{sec:num} for further details. Theorem \ref{thm:l1conv_pdf} is, to the best of our knowledge, a first result in the direction of this paper's main question. Even so, Theorem \ref{thm:l1conv_pdf} is limited in several ways
\begin{enumerate}
\item The demand $|\nabla f|\geq \tau _{\rm f} >0$ is an arbitrary condition from an application standpoint.
\item The differentability and the pointwise derivative-approximation conditions $\|\nabla f - \nabla g \|_{\infty}\lesssim~h^{\tau}$ are strong demands which many other approximation methods do not fulfill. 
\item It is essential that the domain $\Omega$ is compact for the proof to hold.
\item Even when $d=1$, it is required that $d\varrho(\alpha) = c(\alpha) \, d\alpha$ with $c\in C^{1} (\bar{\Omega})$. For comparison, absolute continuity is a weaker condition, as it requires that $c\in C(\Omega) \cap L^1 (\Omega)$. 

\end{enumerate}

The Wasserstein distance (see Sec.\ \ref{sec:wass}) is thus proposed to measure the distance $\mu$ and $\nu$ since it does not suffer from the drawbacks of the norms $\|p_{\mu} - p_{\nu}\|_p $. Admittedly, the $L^p$ distances between the PDFs are both natural in practice and are associated with rich statistical theory; for~$p=1$, then $\|p_{\mu}-p_{\nu}\|_1$ is twice the total variation~\cite{devroye1985l1}, and $\|p_{\mu} - p_{\nu} \|_2 ^2$ is the Integrated Square Error, which is a building block in non-parametric statistics~\cite{ tsybakov2009estimate}. Nevertheless, the analysis of the norms $\|p_{\mu} - p_{\nu}\|_p $ in terms of the functions $f$ and~$g$ can be technically cumbersome; if e.g., $\varrho$ is the Lebesgue measure, then~$p_{\mu}(y)$~is proportional to~$ \int _{f^{-1} (y)} 1/|\nabla f | \, d\sigma$, where $d\sigma $ is the $(d-1)$ dimensional surface measure~\cite{sagiv2018spline}. Moreover, the distance $\|p_{\mu} - p_{\nu} \|_p$ is difficult to work with since it assumes that $\mu$ and $\nu$ have distributions. This is not always the case. For example, let $\varrho$ be the Lebesgue measure on $[0,1]$ and let $$f_k(\alpha) = \left\{ \begin{array}{ll} 0 & x\in [0,\frac{1}{2}] \, , \\
(x-\frac{1}{2})^k  & x\in[\frac{1}{2}, 1] \, ,

\end{array} \right. \qquad k\geq 1 \, . $$
Although $f_k$ is in $C^k([0,1])$, the measure $\mu _k = (f_k) _* \varrho$ is not a absolutely continuous measure and does not have a PDF since $\mu (\left\{0\right\}) = 1/2$.
It is therefore natural to look for other ways to measure the distance between $\mu$ and $\nu$. There are many ways to define distances between probabilities and measures, such as total variation, mutual information, and Kullback-Leibler divergence. The equivalencies and relationships between these norms, metrics, and semi-metrics are the topics of many studies, see e.g.,~\cite{gibbs2002choosing}.

\subsection{The Wasserstein distance}\label{sec:wass}
In order for us to choose the proper metric between $\mu$ and~$\nu$, we revisit Fig.\ \ref{fig:func_pdf}. While the two PDFs seem very different on a local scale, they are quite similar on a coarser scale. For example, $\mu ([0.3,0.4]) \approx \nu ([0.3,0.4])$ and so, if we were to ask what is the probability that the results of many experiments are between $0.3$ and $0.4$, then both $\mu$ and $\nu$ would have provided similar answers. More loosely speaking, since $p_{\nu}$ is oscillatory, the regions where $p_{\nu}>p_{\mu}$ and the regions where $p_{\nu}<p_{\mu}$ are adjacent, and therefore cancel-out each other. The PDF, on the other hand, is the derivative of the measure, and it is therefore heavily affected by local differences. Another disadvantage of the norm $\|p_{\mu} - p_{\nu}\|_q$ is that it does not take geometry into account. Consider for example a family of standard Gaussian measures with mean $t\in \mathbb{R}$, i.e., $p_{\mu ,t}(y) = \exp(-(y-t)^2)/\sqrt{2\pi}$ (see Fig.\ \ref{fig:gaus}). Then for every $t>2$, $\|p_{\mu ,t} - p_{\mu, 0}\|_{1} \approx 2$, regardless of whether $t=3$ or $t=10$ or $t=1,000$.

\begin{figure}[h!]
\centering
{\includegraphics[scale=.5]{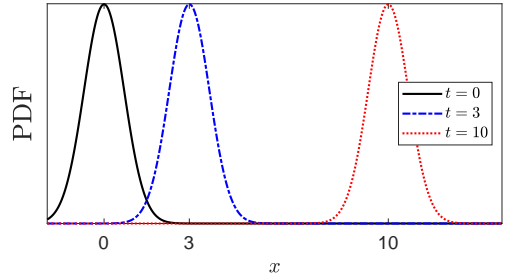}}
\caption{Gaussian distributions centered at $t=0$ (solid), $t=3$ (dash-dots), and $t=10$ (dots). Which of the latter two Gaussians is closer to the $t=0$ Gaussian in Wasserstein distance, and which in the $L^q$ sense?}
\label{fig:gaus}
\end{figure}

A widely-popular metric that overcomes some of the above issues is the Wasserstein metric. Given two probability measures $\mu$ and $\nu$ on $\mathbb{R}$ with $p\geq 1$ finite moments, the Wasserstein distance of order~$p$~is defined as
\begin{subequations}\label{eq:wass}
\begin{equation}
W_p (\mu , \nu) :\,= \left[ \inf \limits_{\gamma \in \Gamma} \int |x-y|^p \, d\gamma (x,y) \right] ^{\frac{1}{p}} \, , 
\end{equation}
where $\Gamma$ is the set of all measures $\gamma$ on $\mathbb{R}^2$ for which $\mu$ and $\nu$ are marginals, i.e.,
\begin{equation}\label{eq:marginals}
\mu (x) = \int\limits_{\mathbb{R}} \gamma(x,y) \, dy \, , \qquad \nu (y) = \int\limits_{\mathbb{R}} \gamma(x,y) \, dx \, .
\end{equation}
\end{subequations}
If the $p$-th moments of $\mu$ and $\nu$ are finite, then a minimizer exists, $W_p(\mu,\nu)$ is finite, and it is a metric \cite{santa2015optimal, villani2003topics}.
Intuitively, the Wasserstein distance with $p=1$ computes the minimal work (distance times force) by which one can transfer a mound of earth that ``looks" like $\mu$ to a one that ``looks" like~$\nu$, and it is therefore referred to as the earth-mover's distance.

As noted, some of the difficulties in approximating the PDFs arise from the inverse proportion between $p_{\mu}$ and $p_{\nu}$ and the {\em gradients} of $f$ and $g$, respectively. It is therefore natural to avoid these issues by considering the integral of the PDF, the cumulative distribution function~(CDF) $$F_{\sigma} (y) :\,= \sigma ([y,\infty)) = \int\limits_{y}^{\infty} p_{\sigma} (t) \, dt \, , $$ for any Borel measure $\sigma$. Indeed, the Wasserstein distance of order $p=1$ is related to the CDF by the following theorem.

\begin{theorem*}[Salvemini \cite{salvemini1943sul}, Vallender \cite{vallender1974calculation}]\label{thm:cdf_1}
For any two probability Borel measures $\mu$ and $\nu$ on $\mathbb{R}$,
$$W_1 (\mu ,\nu) = \int\limits_{\mathbb{R}} |F_{\mu}(x) - F_{\nu} (x)| \, dx \, .$$
\end{theorem*}
This theorem reinforces the notion that $W_1$ is not as sensitive to local effects as $\|p_{\mu} - p_{\nu}\|_p$. Indeed, Fig.~\ref{fig:func_pdf}(c) shows that the two CDFs of $\mu$ and $\nu$ are almost indistinguishable. Furthermore, in the previous Gaussians example (see Fig.\ \ref{fig:gaus}), $W_1 (p_{\mu ,t}, p_{\mu,0})=t$ by direct computation of the CDFs, then, and the same can be proven for $p=2$ as well \cite{givens1984w2, mccann1997w2}. Hence, the geometric distance between the Gaussians matters in under the Wasserstein metric. Generally, Wasserstein distances are a central object in optimal transport theory \cite{santa2015optimal, villani2003topics}, and have also become an increasingly popular in such diverse fields as image processing~\cite{ ni2009local, rubner2000earth}, optimization and neural networks \cite{arjovsky2017was}, well-posedness proofs for partial differential equations with an associated gradient-flow \cite{carrillo2011global}, and numerical methods for conservation laws \cite{solem2018conv, tadmor1991local}. 
\subsection{Structure of the paper} The rest of the paper is organized as follows: Sec.\ \ref{sec:main} presents the main theoretical results of this paper. The upper bounds on $W_p$ (Theorems \ref{thm:wass_gen} and \ref{thm:lq_wp}) are presented in Sec.~\ref{sec:up}, and the lower bounds on $W_1$ (Corollary \ref{cor:w1_low}) and $W_2$ (Theorem \ref{thm:w2_low}) are presented in Sec.\ \ref{sec:low}. The proofs and some technical details of these results are presented in Sec.\ \ref{sec:pfs}. Finally, in Sec.\ \ref{sec:num} the theoretical results are applied to the numerical analysis of uncertainty quantification methods, and a numerical example is presented. %We conclude the paper by elaborating on an optics-related application of this work in Sec.\ \ref{sec:nls_uq}.
\section{Main Results}\label{sec:main}
\subsection{Upper Bounds}\label{sec:up}
In what follows, $\Omega \subseteq \mathbb{R}^d$ is a Borel set, $\varrho$ is a Borel probability measure on~$\Omega $, $f,g:\Omega \to \mathbb{R}$ are measurable, $\mu = f_* \varrho$, $\nu = g_* \varrho$, and $L^p=L^p (\Omega, \varrho)$ for any $1\leq p \leq \infty$ unless stated otherwise.
\begin{theorem}\label{thm:wass_gen}
Let $f$ and $g$ be continuous on $\bar{\Omega}$. 
(i) If $f,g\in L^{\infty}(\Omega, \varrho)$, then for every~$p\geq 1$
$$W_p(\mu, \nu) \leq \|f - g \|_{\infty}  \, .$$
(ii) If $\Omega$ is bounded and $f,g \in L^p (\Omega, \varrho)$ then
$$W_p (\mu , \nu) \leq \|f-g\|_p \, .$$

\end{theorem}

This result is sharp. Let $\varrho$ be any probability measure on $[0,1]$ and let $f(\alpha) \equiv x_0$ and $g(\alpha) \equiv y_0$, for some $x_0,y_0 \in \mathbb{R}$. Then $\mu$ and $\nu$, are the Dirac delta distribution centered at~$x_0$ and $y_0$, respectively, and the only distribution $\gamma\in \Gamma$ is~$\gamma = \delta_{(x_0,y_0)}$. Hence,
$W_p ^p (\mu, \nu)  = |x_0 - y_0|^p =  \|f -g\|_{\infty} ^p$.
Furthermore, as opposed to Theorem \ref{thm:l1conv_pdf}, this theorem does not even demand that $f$ and $g$ be differentiable, and puts no restrictions on the Borel measure $\varrho$. Though this theorem is only valid for domains in $\mathbb{R}^d$, a generalization of case {\em (i)} to (infinite-dimensional) Polish spaces has been achieved by Boussaid \cite{boussaid}.

Item {\em (ii)} of Theorem \ref{thm:wass_gen} uses $L^p$ information to bound $W_p (\mu,\nu)$. In many cases, however, upper bounds on $f-g$ are known only in a specific $L^q$ space. The next theorem shows how $L^q$ error estimates can provide nontrivial upper bounds on $W_p(\mu,\nu)$ for {\em any} $p$, even if $p\neq q$.
\begin{theorem}\label{thm:lq_wp}
Under the assumptions (i)+(ii) of Theorem \ref{thm:wass_gen}, then for every $p,q\geq 1$,
$$W_p (\mu , \nu) \lesssim   \|f-g\|_{\infty} ^{\frac{p}{q+p}}\|f-g\|_q ^{\frac{q}{q+p}} \, ,$$
where $\lesssim$ denotes inequality up to a constant which depend only on $p$ and $q$.
\end{theorem}
This limit agrees with Theorem \ref{thm:wass_gen} in the limit $q\to \infty$ and when $q=p$ (up to a constant). Furthermore, for any $q\neq p$, the bound in Theorem \ref{thm:lq_wp} may improve the $L^{\infty}$ bound in Theorem~\ref{thm:wass_gen}, since $\varrho$ is a probability measure, $(f-g)\in L^{\infty} \cap L^{q}$, and so $\|f-g\|_q  \leq \|f-g\|_{\infty}$.

\subsection{Lower bounds}\label{sec:low}
The $W_1$ lower bound is the direct result of the Monge-Kantorovich duality, see Sec.\ \ref{sec:l1_low_pf} for details and proof.
\begin{corollary}\label{cor:w1_low}
If $f,g\in  C(\bar{\Omega})$ and $\Omega$ is bounded, then $$\left| \mathbb{E}_{\varrho} f - \mathbb{E}_{\varrho} g \right| \leq W_1 (\mu , \nu) \leq \|f-g\|_{L^1(\Omega, \varrho)} \, .$$
Moreover, if $f\geq g$ almost everywhere with respect to $\varrho$, then $$W_1 (\mu , \nu ) = \|f-g\|_{L^1(\Omega, \varrho)} \, .$$
\end{corollary}
We note that since the upper bound is sharp (see discussion on Theorem \ref{thm:wass_gen}) and since equality might hold, the lower bound is sharp too. We further note that in the case where $\Omega$ is the unit circle, lower bounds on $W_1$ in terms of the Fourier coefficients of $f$ were proved by Steinerberger~\cite{steinerberger2018was}.

Next, to bound $W_2 (\mu,\nu)$ from below, we introduce two concepts: the Sobolev space $\dot{H}^{-1}$ and the symmetric decreasing rearrangement. For any Borel measure $\sigma$ on~$\mathbb{R}$, define the semi-norm $$\|\sigma \|_{\dot{H}^{-1}(\mathbb{R})} :\,= \sup\limits_{\|q\|_{\dot{H}^1 (\mathbb{R}) } \leq 1} |\langle q, \sigma \rangle | \, ,$$ where $\|q\|_{\dot{H}^1} ^2 = \int |q'(x)|^2 \, dx$ \cite{adams2003sobolev}. Note that $\|\sigma \|_{\dot{H}^{-1} } < \infty$ only if $\sigma (\mathbb{R})=0$. Another way to understand the Sobolev semi-norm $\dot{H}^{-1}$ and to compare it to the more frequently used $L^2$ norm is through Fourier analysis. By Plancharel Theorem $$ \|\sigma \| ^2_{L^2} = \int _{\mathbb{R}} \left|\hat{\sigma}(\xi)\right|^2 \, d\xi \, , \qquad \|\sigma \|_{\dot{H}^{-1}} ^2 = \int _{\mathbb{R}} \left|\frac{\hat{\sigma}(\xi)}{|\xi|}\right|^2 \, d\xi \,  ,$$ 
where $\hat{\sigma}$ is the Fourier transform of $\sigma$ \cite{adams2003sobolev}. Thus, if $\mu$ and $\nu$ are different only in high frequencies, then their $L^2$ difference might be much higher than their $\dot{H}^{-1}$ difference (due to the $1/|\xi|$ term in the integral). Intuitively, it means that highly local effects in $\sigma =\mu-\nu$ are ``subdued" in the negative Sobolev semi-norm. This is analogous to the way local effects in the PDFs are subdued in the $W_1$ distance, i.e., in the CDFs (see Fig.~\ref{fig:func_pdf}). As noted, this property also characterizes the Wasserstein distance, and indeed Loeper \cite{loeper2006uniqueness} and Peyre \cite{peyre2018non} related $W_2(\mu ,\nu)$ to $\|\mu -\nu \|_{\dot{H}^{-1} }$ in the following theorem:
\begin{theorem*}[Loeper \cite{ loeper2006uniqueness}, Peyre \cite{peyre2018non}]\label{thm:peyre}
Let $\mu$ and $\nu$ be probability measures on $\mathbb{R}$ with densities $p_{\mu}, p_{\nu} \in L^{\infty}(\mathbb{R})$, respectively. Then,
$$\|\mu - \nu\|_{\dot{H}^{-1}} \leq \max\left\{\|p_{\mu} \|_{\infty} , \|p_{\nu} \|_{\infty}\right\}^{\frac{1}{2}} W_2(\mu, 
\nu ) \, ,$$
\end{theorem*}

\begin{figure}[h!]
\centering
{\includegraphics[scale=.5]{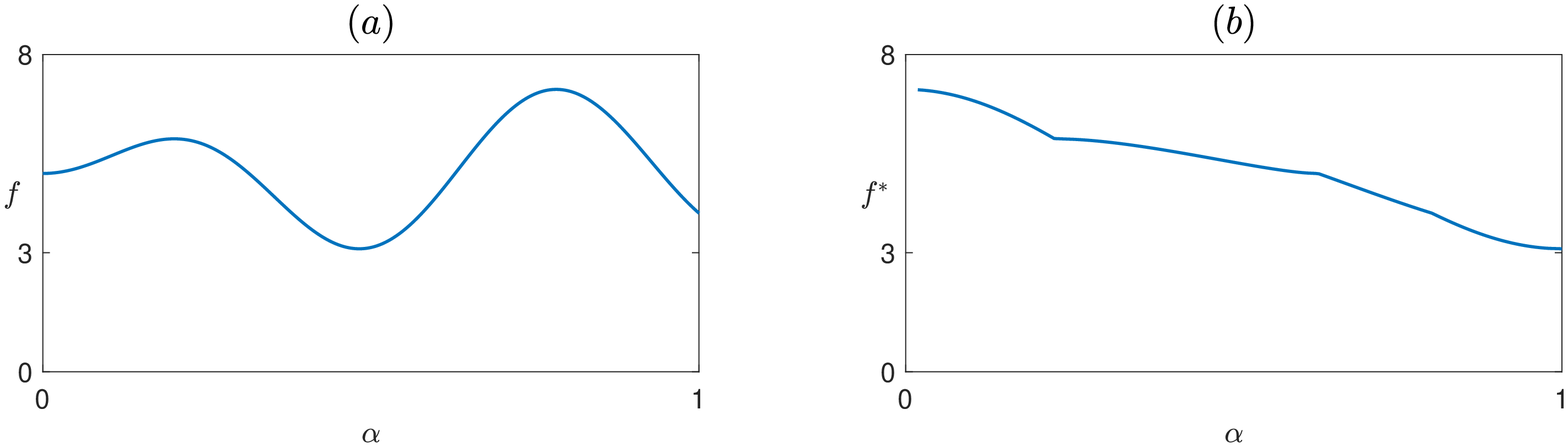}}
\caption{(a) $f(\alpha)= 5(1+\alpha \sin (10\alpha) e^{-\alpha^2})$. (b) $f^*(\alpha)$, the symmetric decreasing rearrangement of $f$, with respect to the Lebesgue measure on [0,1].}
\label{fig:decreasing}
\end{figure}

We now introduce the Symmetric decreasing rearrangement by an absolutely-continuous Borel probability measure on $\Omega \subseteq \mathbb{R}^d$ \cite{leeb2001analysis}. The symmetric decreasing rearrangement of a measurable set $A$ is $$ A^\star = \{ \alpha \in \Omega ~ |~ \varrho (B(0,1))\cdot |\alpha |^d \leq \varrho (A) \} \, ,$$
where $B(0,1)\subset \mathbb{R}^d$ is the unit ball around the origin. Next, for a measurable non-negative function $f:\mathbb{R}^d \to \mathbb{R}_+$, define the symmetric decreasing rearrangement as $$ f^{\star} (\alpha) = \int\limits_{0}^{\infty} \mathbbm{1}_{\{\alpha'\in \Omega ~|~ f(\alpha') > t \}^{\star}} (r) \, dt , \qquad r:\,=|\alpha| \, ,$$
where $\mathbbm{1}_B$ is the identifier of a set $B\subseteq \mathbb{R}^d$. For a numerical example of the symmetric decreasing rearrangement, see Fig.\ \ref{fig:decreasing}. In more intuitive terms, $f^*$ is the unique monotonic decreasing function such that $\varrho(A(f,t))=\varrho(A(f^*,t))$ for all $t\in \mathbb{R}$, where $A(f,t) :\,=~\{\alpha ~{\rm s.t.}~ f(\alpha)\geq~t \}$ are the super-level sets of $f$. Moreover, since $f^*$ is monotonic decreasing, one also have that $A(f^* ,t)$ is the interval $ [0,\varrho(A(f,t))]$. The symmetric decreasing rearrangement is an important object in real analysis \cite{leeb2001analysis}, with notable properties such as the P{\'o}lya-Szego inequality \cite{polya1951iso} $$\|f\|_p = \|f^*\|_p  \, , \qquad \|\nabla f^*\|_p \leq \|\nabla f \|_p  \, ,$$ for all $p\geq 1$. Hence, the symmetric decreasing rearrangement $f^*$ minimizes all Sobolev $W^{1,p}$ norms among the class of functions with the same super-level sets, it can be said to be the ``canonical" representative this class.

\begin{theorem}\label{thm:w2_low}
Let $I$ be a closed and bounded interval equipped with an absolutely-continuous probability measure $\varrho$ with a bounded and continuous weight function $r(\alpha)$, i.e., $d\varrho (\alpha) = r(\alpha) d\alpha$, and let $f,g\in C^1$ with $|(f^*)'|, |(g^*)'| > \tau >0$. Then, for every $k\in \mathbb{N}$ $$W_2(\mu, \nu) \geq A_k \left| \mathbb{E}_{\varrho} f^k- \mathbb{E}_{\varrho} g^k \right|\, ,$$
where $A_k$ is a positive coefficient given by $$A_k = A_k(f,g,r) = \frac{\sqrt{2k -1}}{k} \left(\max(f,g)^{2k-1}-\min(f,g)^{2k-1}\right)^{-\frac{1}{2}}  \tau ^{\frac{1}{2}} \|r\|_{\infty}^{-\frac{1}{2}} \,, $$ 
and the $\max$ and $\min$ are taken over all $x\in I$.
\end{theorem}

We remark that even though $A_k = A_k (f,g,r)$ depends on $f$ and $g$, it does not depend directly on $f-g$. Hence, for a sequence $(g_n(\alpha))_{n=1} ^{\infty}$ which converges uniformly to $f$, for each $k\in\mathbb{N}$, then $A_k(f,g_n,r)$ would converge to a positive constant as $n\to \infty$. A specific example to the computation of the coefficients $A_k$ can be found in Sec.\ \ref{sec:w2_low}.

\section{Proofs of main results and technical discussion}\label{sec:pfs}
\subsection{Proof of Theorem \ref{thm:wass_gen}}\label{sec:pf_wass}

\begin{proof}
We begin with the case where $f$ and $g$ are {\em uniformly continuous} in $\Omega$. Let $\epsilon >0$, then by uniform continuity there exists $\eta = \eta (\epsilon) >0$ such that $|f(\alpha) - f(\beta)|<\epsilon$ and $|g(\alpha ) - g(\beta)|<\epsilon$ for every $\alpha,\beta \in \Omega$ such that $|\alpha -\beta |<\eta $. Let $L\in \mathbb{N}$ and partition~$[-L,L]^d$ to $M$ equal-size boxes $\{ \tilde{I}_j\}_{j=1}^M$ such that ${\rm diam} (\tilde{I}_j)<\eta$. Let $I_j = \tilde{I}_j \cap \Omega$ for every $1\leq j \leq M$ and let $I_{M+1} :\,= \Omega \setminus [-L,L]^d$. Next, let
$$ \mu _j :\,= f_* \varrho\big|_{I_j} \, , \qquad \nu _j :\,= g_* \varrho\big|_{I_j} \, ,$$
i.e., the measures induced by $f(I_j)$ and $g(I_j)$ for every $1\leq j \leq M+1$. Since $\int_{\mathbb{R}} \mu _j = \int_{\mathbb{R}}\nu _j = \varrho (I_j)$, we can transport $\mu$ to $\nu$ by transporting each $\mu _j$ to~$\nu _j$. Even though this might not be the optimal transport between $\mu$ and $\nu$, since $W_p$ is defined as an infimum over all transports then \begin{subequations}\label{eq:subinterval_wass}
\begin{equation}W_p ^p (\mu , \nu) \leq \sum\limits_{j=1}^{M} W_p ^p (\mu _j, \nu _j)  + W_p ^p (\mu _{M+1}, \nu_{M+1}) \, ,
\end{equation}
for any $1\leq j \leq M+1$, where
\begin{equation}
\begin{aligned}
W_p ^p (\mu _j , \nu _j )  :\,&= \inf\limits_{\gamma \in \Gamma_j}  \int\limits_{f(I_j)\times g(I_j)} |x-y|^p \,d\gamma(x,y)\\
&\leq \left(\sup\limits_{(x,y)\in f(I_j)\times g(I_j)}|x-y|^p \right) \varrho(I_j) \, , 
\end{aligned}
\end{equation}
\end{subequations}
where $\Gamma _j$ is the set of all measures whose marginals are $\mu _j$ and $\nu _j$. For $1\leq j \leq M$, since ${\rm diam} (I_j) < \eta$, then by uniform continuity for any $t\in I_j$ $$\sup\limits_{(x,y)\in f(I_j)\times g(I_j)}|x-y|^p \leq \left(|f(t)-g(t)|+2\epsilon \right)^p \, .$$

Here the proofs of the $L^p$ and $L^{\infty}$ bounds slightly diverge and we begin with proving that $W_p(\mu,\nu)\leq \|f-g\|_{\infty}$. For any $1\leq j \leq M$ then $ (|f(t)-g(t)|+2\epsilon )^p \leq \left(\|f-g\|_{\infty} + 2\epsilon \right)^p$ . Similarly, for $j=M+1$, the supremum in \eqref{eq:subinterval_wass} is bounded from above by $\left( \|f\|_{\infty} + \|g\|_{\infty} \right)^p$.
Combining these bounds together, we have that
$$W_p ^p (\mu , \nu) \leq \|f-g\|_{\infty} ^p\sum\limits_{j=1}^M  \varrho(I_j) +o(\epsilon)\sum\limits_{j=1}^M  \varrho(I_j) +\left( \|f\|_{\infty} + \|g\|_{\infty} \right)^p \varrho(I_{M+1})  \, . $$
Since $\varrho$ is a probability measure $\sum _{j=1} ^M \varrho(I_j) = \varrho(\Omega) = 1$ and as $L \to \infty$  the third term on the right-hand-side vanishes. Hence $W_p ^p (\mu ,\nu) \leq \|f-g\|_{\infty} ^p +2o(\epsilon )$ for every $\epsilon >0$, and so $W_p  \leq \|f-g\|_{\infty}$.

Next, consider the case where $f,g$ are continuous on $\Omega$, but not uniformly continuous. Then for any two sequences $\epsilon _n \to 0$ and $L_n \to \infty$, choose $\eta _n = \eta (\epsilon _n ,L_n)$ which satisfies the uniform continuity condition on the compact domain $\bar{\Omega} \cap [-L_n,L_n]$. Then, by partitioning this domain to sufficiently many boxes $M_n=M(\eta_n )=M(\epsilon_n , L_n)$  such that ${\rm diam} (I_{j,n} ) \leq \eta _n$, the proof holds as~$n\to \infty$.

Finally, we prove that $W_p (\mu,\nu) \leq \|f-g\|_p$. Here we require that $\Omega$ is bounded, and so we can choose $L$ such that $\Omega \subseteq~[-L,L]^d$. For $1\leq j \leq M$ we have that for some $t_j \in I_j$ then
\begin{align*}
\sup\limits_{(x,y)\in f(I_j)\times g(I_j)} |x-y|^p &\leq \left(|f(t_j)-g(t_j)| +2\epsilon\right)^p \\
&= |f(t_j)-g(t_j)|^p +o(\epsilon) \, .
\end{align*}
Substituting this inequality in \eqref{eq:subinterval_wass} yields $$W_p ^p (\mu, \nu) \leq  \sum\limits_{j=1}^M  |f(t_j)-g(t_j)|^p \varrho(I_j)  +o(\epsilon) \sum\limits_{j=1}^M \varrho(I_j) \, .$$
As the partition is refined (i.e., $M\to \infty$ and $\eta \to 0$), the first element on the right-hand-side converges to $\|f-g\|_{L^p (\varrho)}$. Since $\varrho$ is a probability measure, $\sum _{j=1}^M \varrho (I_j) = 1$, and so the second element on the right-hand-side is $o(\epsilon)$. Since this inequality is true for {\em any} $\epsilon>0$, the proof follows.

\end{proof}

\subsection{Proof of Theorem \ref{thm:lq_wp}}\label{sec:lq_wp_pf}
\begin{proof}
Define $\Omega _{r} :\,= \{ \alpha \in \Omega ~ |~ |f(\alpha)-g(\alpha)| \geq r \}$ for any $r >0$, and let $\mu _{\Omega_r}$, $\mu _{\Omega \setminus \Omega_r}$, $\nu _{\Omega_r}$, and $\nu _{\Omega \setminus \Omega_r}$ be the measure induced by $f(\Omega _r)$, $f(\Omega \setminus \Omega _r)$, $g(\Omega _r)$, and $g(\Omega \setminus \Omega _r)$, respectively. For any~$p\geq 1$,
\begin{equation}\label{eq:wp_break_lq}
W_p ^p (\mu,\nu) \leq W_p ^p (\mu _{\Omega _r} , \nu _{\Omega _{r}}) + W_p ^p (\mu _{\Omega \setminus \Omega _r } , \nu _{\Omega \setminus \Omega _{r}}) \, .
\end{equation}
The fist term on the right-hand-side of \eqref{eq:wp_break_lq} is bounded from above by $\|f-g\|_{\infty}^p\varrho(\Omega _r)$, due to Theorem \ref{thm:wass_gen}. To bound~$\varrho(\Omega _r)$, note that $$ \|f-g\|_{L^q (\Omega)} ^q \geq \|f-g\|_{L^q (\Omega_{r})} ^q \geq \varrho (\Omega _{r} ) \cdot r^q  \, ,$$
where the first inequality is due to monotonicity of $\varrho$, and the last inequality is due the continouity of $|f-g|^q$. Hence, $\varrho (\Omega_r) \leq \|f-g\|_q ^q r^{-q}$, and so the first term in the right-hand-side of \eqref{eq:wp_break_lq} is bounded from above by~$\|f-g\|_{\infty}^p  \|f-g\|_q ^q r^{-q}$.  Since the $L^{\infty}$ upper bound of Theorem \ref{thm:wass_gen} is applicable to $f$ and $g$, and since $\varrho\left( \Omega \setminus \Omega _r \right) \leq 1$, then the second term on the right-hand-side of~\eqref{eq:wp_break_lq}~is bounded from above by $\|f-g\|_{L^{\infty} (\Omega \setminus \Omega_r )} ^p \leq  r^p$. Having bounded from above both terms on the right-hand-side of \eqref{eq:wp_break_lq}, then $$W_p ^p (\mu , \nu ) \leq \|f-g\|_{\infty}^p \|f-g\|_q ^q r^{-q} + r^p \, .$$
To minimize the right-hand-side of this inequality, we derive with respect to $r$ and get that the minimum is achieved at $r_{\min} = ( qp^{-1} \|f-g\|_q ^q  \cdot \|f-g\|_{\infty}^p ) ^{1/(p+q)}$, and so 
\begin{align*}
W_p (\mu , \nu) &\leq \left[\|f-g\|_{\infty} ^p \|f-g\|_q ^q r_{\min} ^{-q} + r_{\min} ^p \right]^{\frac{1}{p}} \\ &\lesssim \|f-g\|_{\infty}^{\frac{p}{q+p}}  \|f-g\|_{q}^{\frac{q}{q+p}} \, . 
\end{align*}

\end{proof}

\subsection{Proof of Corollary \ref{cor:w1_low} }\label{sec:l1_low_pf}
\begin{proof}
The Monge-Kantorovich duality states that \cite{villani2003topics}
$$W_1 (\mu ,\nu) = \sup \left\{\left| \int\limits_{\mathbb{R}} w(y) \, (d\mu (y) - d\nu (y))\right| ~~: ~~ L
(w) \leq 1 \right\} \, ,$$

where $L(w)$ is the Lipschitz constant of $w$. So, to prove a non-trivial lower bound for $\mu = f_* \varrho$ and $\nu =g_* \varrho$, it is sufficient to provide {\em any function} $w$ for which the integral is not zero. Let $w(y) = y$. Since $L(w)=1$, then $W_1(\mu , \nu) \geq | \int_{\mathbb{R}} y \, d\mu (y) - \int_{\mathbb{R}} y \, d\nu (y) |$, which, by change of variables, means that $W_1 (\mu ,\nu ) \geq | \int_{\Omega} (f(\alpha)-g(\alpha)) \, d\varrho (\alpha) | $. Combined with Theorem \ref{thm:wass_gen} we arrive at the corollary.
\end{proof}
\subsection{Proof of Theorem \ref{thm:w2_low}}\label{sec:w2_low}
\begin{proof}
By definition of the symmetric decreasing rearrangement, $\mu = f^* _* \varrho$ and $\nu = g^*  _* \varrho$. Moreover, since the theorem requires that $|(f^*)'|, |(g^*)'|>\tau :\, >0$, we can assume without loss of generality that $f$ and $g$ are strongly monotonically decreasing. Next, we have the following standard lemma (for proof, see e.g., \cite{sagiv2018spline}):
\begin{lemma*}
Let $h\in C^1(I)$ be piecewise monotonic, let $d\varrho(\alpha) = r(\alpha) d\alpha$ where $r$ is continuous in $\Omega$. Then the PDF of the measure $\sigma = h_* \varrho$ is given by $$p_{\sigma}(y)~=~\sum\limits_{\alpha \in h^{-1} (y)} \frac{r(h^{-1}(y))}{|h'(h^{-1} (y))|} \,  , \qquad  y \in {\rm range} (h) \, .$$
\end{lemma*}
Hence, by definition and the above lemma
\begin{align*}
\|\mu - \nu\|_{\dot{H}^{-1}} &= \sup\limits_{\|q\|_{\dot{H}^{1}} \leq 1}\int\limits_{\mathbb{R}} q(y)\, (p_{\mu}(y) - p_{\nu} (y)) \, dy \\
&= \sup\limits_{\|q\|_{\dot{H}^{1}} \leq 1}\left| \int\limits_{\mathbb{R}} q(y) \frac{r(f^{-1}(y))}{f'(f^{-1}(y))} \, dy - \int\limits_{\mathbb{R}} q(y) \frac{r(g^{-1}(y))}{g'(g^{-1}(y))} \, dy\right| \, .
\end{align*}
Consider the first integral under the supremum. By change of variables $y= f(x)$, we have 
\begin{align*}
\int\limits_{\mathbb{R}} q(y) \frac{r(f^{-1}(y))}{f'(f^{-1}(y))} \, dy &= \int\limits_{I} q\circ f (x) \frac{r(x)}{f'(x)} f'(x) \,dx\\
&=\int\limits_{I} q\circ f \, d\varrho (x) \,. 
\end{align*}
Doing the respective change of variable for the second integral under the supremum, we have 
\begin{equation}\label{eq:sup_hmin1}
 \|\mu - \nu\|_{\dot{H}^{-1}}  =\sup\limits_{\|q\|_{\dot{H}^{1}} \leq 1}\left| \int\limits_I \left( q\circ f (x) - q\circ g (x)\right) \, d\varrho(x) \right|  .
\end{equation}

For ease of notations, denote $M=\max_{x\in I} \left\{f(x),g(x) \right\}$ and $m =\min_{x\in I} \left\{f(x),g(x) \right\}$. Since $f$ and $g$ are continuous on a closed bounded interval, both $m$ and $M$ are finite. Fix $k\in \mathbb{N}$, and let $q_k (x) = c_k x^k $, where the normalization constant $c_k :\, = (\sqrt{2k-1}/k)(M^{2k-1} - m^{2k-1} )^{-1/2}$ is chosen so that $\|q_k \|_{\dot{H}^1[m,M]} = 1$.\footnote{It might seem that the choice of the interval $[m,M]$ is made ad-hoc. However, this proof can be carried out in the space $\dot{H}^{-1}(\mathbb{R})$ regardless, by the following construction: extend $q_k(y)$ to $\mathbb{R}$ by setting $q_k(y) = q_k(m)$ for $y<m$ and $q_k(y)=q_k(M)$ for $y>M$. Since outside $[m,M]$, $q_k' \equiv 0$, then $\|q_k\|_{\dot{H}^{1}(\mathbb{R})} =\|q_k\|_{\dot{H}^{1}([m,M])}$, and $\langle q_k, \mu-\nu \rangle$ is unchanged too since $\mu-\nu$ is supported only on $[m,M]$. Our choice is also consistent with the result by Peyre \cite{peyre2018non}, since these also "take place" on the supports of $\mu$ and $\sigma$.}  Hence, substituting $q_k$ in~\eqref{eq:sup_hmin1}~for every $k\in \mathbb{N}$
\begin{align*}
\| \mu - \nu \| _{\dot{H}^{-1}}  &\geq \left| \int\limits_{I} (q_k \circ f(x) - q_k\circ g(x)) \, d\varrho (x) \right| \\
&= c_k \left|\int\limits_{I} f^{k} (x) - g^{k} (x) \, d\varrho (x) \right| \\
&= c_k \left| \mathbb{E}_{\varrho} f^k - \mathbb{E}_{\varrho} g^k \right | \, .
\end{align*}

Finally, to bound $W_2$ from below we need Loeper and Peyre's theorem, and so we need to compute $\|p_{\mu}\|_{\infty}$ and $\|p_{\nu} \|_{\infty}$. As noted, since $f=f^*$ is strictly decreasing, it is also continuously differentiable almost everywhere. Hence. by the result noted above, $p_{\mu}= r(f^{-1}(y))/|f'(f^{-1} (y))|$ almost everywhere, and so $\|p_{\mu}\|_{\infty}
\leq \tau  ^{-1} \|r\|_{\infty}$. Since the same holds for $g$ and $\nu$ as well, we substitute in Loeper's and Peyre's bound and get that
\begin{align*}
W_2 (\mu, \nu) &\geq [\max \{ \| p_{\mu} \|_{\infty} , \| p_{\nu} \| _{\infty} \}]^{-\frac{1}{2}} \|\mu - \nu \|_{\dot{H}^{-1}}\\
& \geq [\tau  \|r\|_{\infty} ^{-1}]^{\frac{1}{2}}\|\mu - \nu \|_{\dot{H}^{-1}} \\
&\geq \tau ^{\frac{1}{2}} \|r\|_{\infty}^{-\frac{1}{2}}  c_k \left| \mathbb{E}_{\varrho} f^k - \mathbb{E}_{\varrho} g^k \right | \, .
\end{align*}

\end{proof}

We complement the proof by an example of a direct computation of the coefficients $A_k$. Let $f(\alpha)=3\alpha -3$, $g(\alpha) = 2\alpha - 2$ and $\varrho$ is the Lebesgue measure on~$[0,1]$, then by direct computation we have that $M=0$, $m=-3$, $\|r\|_{\infty}=1$, $\tau = 2$, and so
$$A_k = \frac{\sqrt{2k+1}}{k}3^{-k+\frac{1}{2}}2^{\frac{1}{2}} \cdot 1 \,  , \qquad k\in \mathbb{N} \, .$$

\section{Convergence of uncertainty-quantification methods and numerical examples}\label{sec:num}

We apply the main theoretical results of this paper to the analysis of uncertainty quantification~(UQ)~methods. In many applications, one can only compute the quantity of interest~$f(\alpha)$ for a finite subset of $\alpha$ values~$\{\alpha_j \}_{j=1}^N$. To compute $\mu = f_* \varrho$, we first use these sampled values $\{f(\alpha_j)\}_{j=1}^N$ to construct an approximate function $g(\alpha)$, and then we approximate $\mu \approx \nu = g_* \varrho$, see Fig.~\ref{fig:problem}.
This measure-approximation problem is characterized by the following trade-off: The computational cost comes from direct computation of the samples $\{f(\alpha _j)\}_{j=1}^N$, \footnote{Since $g$ is given in closed form, e.g., by a polynomial, it is computationally cheap to estimate the measure $\nu = g_* \varrho$. Computing $f(\alpha _j)$, on the other hand, might involve a full numerical solution of a PDE.} and so it {\em increases} linearly with $N$. On the other hand, we expect the approximation error to {\em decrease} with the sample size~$N$, i.e., as we improve the sampling resolution. The question is, therefore, how to construct~$g$~such that $\mu$ is accurately approximated with a small sample size $N$.

In terms of numerical analysis, the main result of this paper is that upper bounds on $\|f-g\|_q$ {\em do guarantee} an upper bound on the Wasserstein distances $W_p(\mu, \nu)$. This in turn immediately implies an upper bound on the $L^1$ distance between the CDFs, due to the previously-noted Salvemini-Vallender identity $W_1 (\mu, \nu) = \|F_{\mu} - F_{\nu} \|_1$ \cite{ vallender1974calculation}. 

The upper bounds on the Wasserstein-error stand in sharp contrast to the $L^q$ errors between the PDFs, since in general an upper bound on $\|f-g\|_q$ does not guarantee an upper bound on $\|p_{\mu}-p_{\nu}\|_{L^p}$, for any finite $p$ and $q$ \cite{sagiv2018spline}. We therefore see that the way we define the approximation-error in this problem is not a mere technicality, but rather determines the results of the convergence analysis. Furthermore, we see that CDFs are ``easier" to approximate than PDFs, in the sense that the it is easier to guarantee their efficient approximation.  

We demonstrate the applicability of our theory for two approximation methods (surrogate models), spline interpolation and generalized Polynomial Chaos~(gPC).

\subsection{Spline interpolation.}  Given an interval $\Omega = [\alpha_{\min},\alpha_{\max}]$ and grid-points $\alpha_{\min} = \alpha_{1} < \alpha _2 < \cdots <\alpha _N = \alpha_{\max}$, an interpolating $m$-th order spline $g(\alpha) \in C^{m-1}(\Omega)$ is a piecewise polynomial of order $m$ that interpolates~$f(\alpha)$ at the grid-points, endowed with some additional boundary conditions so that it is unique. See \cite{deboor1978splines, prenter2008splines} for comprehensive expositions on splines, see~\cite{rice1978tensor, schultz1969spline} for their extension to multidimensional domains via tensor-products, and see \cite{beck2019iga, halder2018adaptive} for their applicability to UQ problems. Since Theorem \ref{thm:spline_was} is directly applicable to spline interpolation \cite{sagiv2018spline}, if $g$ is the spline interpolant of $f$, then the PDFs of $\mu$ and $\nu$ are close, i.e., $\|p_{\mu}-p_{\nu} \|_{L^p}$ is bounded from above for any $1\leq p < \infty$. We show that in these settings, the Wasserstein distance between the measures is also bounded from above.	

\begin{theorem}\label{thm:spline_was}
Let $f\in C^{m+1}([0,1]^d)$, let $g(\alpha)$ be its (tensor-product) spline interpolant of order $m$ on a (tensor-product) grid of maximal grid size $h$, and let $\varrho$ be a probability Borel measure. Then, for every $p\geq 1$, $$W_{p}(\mu,\nu ) \lesssim h^{m+1} \approx N^{-\frac{m+1}{d}}   \, , \qquad \|F_{\mu} - F_{\nu}\|_1 \lesssim N^{-\frac{m+1}{d}} \, ,$$
where $N$ is the total number of interpolation points, and where $\lesssim$ and $\approx$ denote inequality and equality up to constants independent of $h$ and $N$, respectively.
\end{theorem}
See Sec.\ \ref{sec:spline_pf} for the proof. Theorem \ref{thm:spline_was} is stronger than Theorem \ref{thm:l1conv_pdf} in three aspects. First, Theorem~\ref{thm:spline_was}~holds for a broader function class than the application of Theorem \ref{thm:l1conv_pdf} to splines, since it does not require that $|\nabla f| > \tau _{f} > 0$, or even that the underlying measure~$\varrho$~would be absolutely continuous. Second, Theorem \ref{thm:spline_was} is non-trivial even for those functions for which Theorem~\ref{thm:l1conv_pdf} does apply. To obtain a ``trivial" upper bound, note that for any two probability measures of $\mu$ and $\nu$ with PDFs $p_{\mu}$ and $p_{\nu}$, then $$W_1 (\mu ,\nu) \leq \frac{1}{2}{\rm diam} (\Omega)\cdot \|p_{\mu} - p_{\nu} \|_1 \, ,$$
where ${\rm diam}(\Omega)$ is the diameter of $\Omega = {\rm supp} (\mu) \cup {\rm supp} (\nu)$ \cite{gibbs2002choosing}. Since $f$ and $g$ are continuous on a compact set, they are bounded, and so the supports of $\mu$ and $\nu$ are bounded as well. Hence, ${\rm diam} (\Omega)<\infty$, and so by Theorem \ref{thm:l1conv_pdf}, $W_1 (\mu ,\nu) \leq Kh^{m}$. Theorem \ref{thm:spline_was}, however, guarantees an additional order of accuracy and so non-trivially improves the previous results.\footnote{Unfortunately, Theorem \ref{thm:spline_was} cannot improve the $L^1$ bound in Theorem~\ref{thm:l1conv_pdf} since, in general, $\|p_{\mu} - p_{\nu} \|_1 \lesssim  W_1(\mu , \nu)$ only for finite spaces~\cite{gibbs2002choosing}. } Finally, Theorem~\ref{thm:spline_was} applies not only for $p=1$ but for all $p\geq1$.

\begin{figure}[h!]
\centering
{\includegraphics[scale=.5]{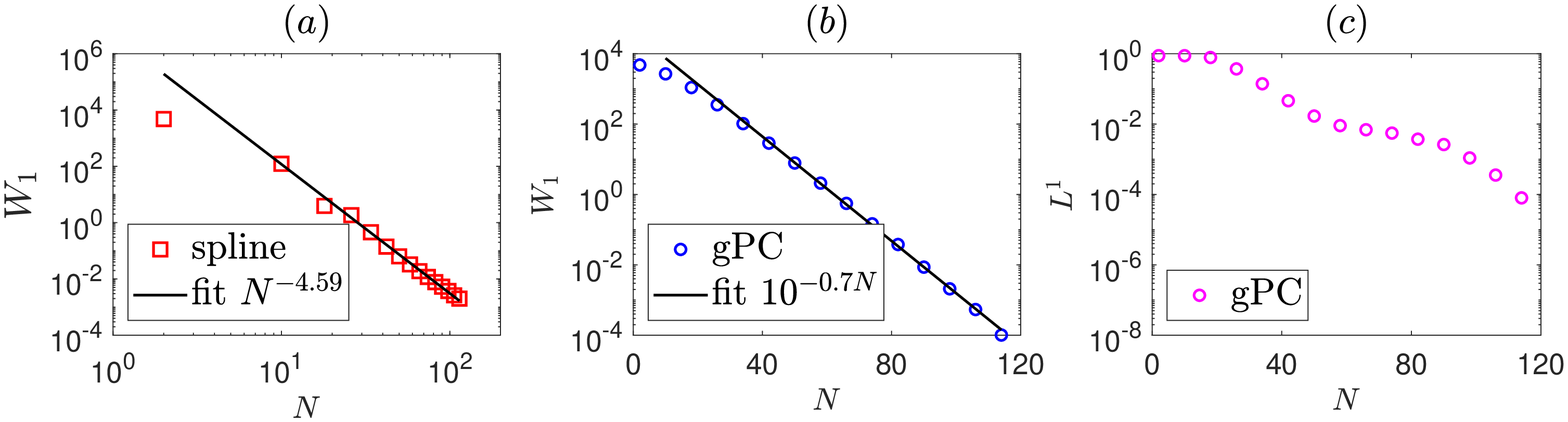}}
\caption{Approximation of $\mu = f_* \varrho$ where $f(\alpha)= \alpha/2 + \tanh(9\alpha)$ and~$\varrho $ is the uniform probability measure on $[-1,1]$. (a) $W_1(\mu, \nu)$ where $g$ is the spline interpolant of $f$ on a uniformly spaced grid (rectangles) and a polynomial fit~$\sim~N^{-4.59}$~(solid),  as predicted by Theorem \ref{thm:spline_was}. (b) Same, but where $g$ is the Collocation gPC approximation of $f$ (circles) and an exponential fit~$\sim~10^{-0.7N}$~(solid), as predicted by Theorem \ref{thm:gpc}. (c)~$L^1$ error of the PDFs using the collocation gPC method.}
\label{fig:w1_tanh}
\end{figure}

\paragraph{\bf Numerical example.} Let 
\begin{equation}\label{eq:f_numer}
f(\alpha)=\frac{\alpha}{2}+\tanh(9\alpha) \, , \quad \Omega = [-1,1] \, ,\quad  d\varrho(x) = \frac{1}{2}\,dx \, .
\end{equation}
We use a cubic spline interpolant on a grid of $N$ uniformly-spaced points, with the not-a-knot boundary condition \cite{deboor1978splines}. Theorem \ref{thm:spline_was} guarantees that in this case $W_p(\mu, \nu) \lesssim N^{-4}$. Indeed, Fig.~\ref{fig:w1_tanh}(a) shows the $W_1$ difference between the two measures as a function of $N$, and that the convergence rate is $N^{-4.59}$. 

\subsection{Generalized Polynomial Chaos (gPC)} Next, we turn to study $W_p$ convergence of $L^2$-spectral methods, for which PDF convergence is an open problem. We focus on the widely popular generalized Polynomial Chaos (gPC). 

{\bf Review of the Collocation gPC method.}  For a more detailed exposition, see e.g.,  \cite{gottlieb1977numerical, xiu2005colocation}.
Let the Jacobi polyomials $\{ p_n (x) \}_{n=0} ^{\infty }$ be the orthogonal polynomials with respect to $\varrho$, i.e., $p_n$ is a polynomial of degree $n$, and $\int _{-1}^{1} p_n (\alpha) p_m(\alpha) \, d\varrho(\alpha) = \delta _{n,m}$, see \cite{ szego1939orthogonal}  for details. This family of orthogonal polynomials constitutes an orthonormal basis of the space $ L^2(\Omega, \varrho)$, i.e., for every $f\in L^2$ one can expand
$$
f(\alpha ) = \sum\limits_{n=0}^{\infty } \hat{f} (n) p_n (\alpha ) \, , \qquad \hat{f}(n) := \int\limits_{\Omega} f(\alpha) p_n (\alpha)  \,d\varrho (\alpha) \, .$$
This expansion converges spectrally, i.e., if $f$ is in $C^{r}$, then $\{\hat{f} (n)\} \lesssim n^{-r}$, and if $f$ is analytic in an ellipse $E\subseteq \mathbb{C}$ that contains $[-1,1]$, then $|\hat{f} (n)| \lesssim e^{-\gamma n}$, for some $\gamma >0$. Thus, one has that  for such analytic functions $$\|f-\pi _N (f) \|_2 \lesssim e^{-\gamma N}  \, , \qquad \pi_N (f) :\,= \sum _{n=0}^N \hat{f}(n) p_n (\alpha) \, . $$

The expansion coefficients~$\{ \hat{f} (n)\}$ can be approximated using the Gauss quadrature
$$
\hat{f} (n)  \approx \hat{f}_N (n) :\,= \sum\limits_{j=1}^N f\left( \alpha _j  \right) p_n  \left( \alpha _j  \right) w_j ,\qquad n=0,1,\ldots, N-1 \, ,
$$
where~$\{ \alpha _j \}_{j=1}^{N}$ are the quadrature points, the distinct and real roots of $p_N(\alpha)$, and $w_j$ are the quadrature weights \cite{davis1967integration}. 
We define the gPC Collocation approximation $g_N$ to be the truncated expansion of $f$ with the quadrature-based coefficients $\hat{f}_N (n)$. We remark that this approximation method has a much simpler form --
The gPC collocation approximation is also the unique interpolating polynomial of $f$ of order $N-1$ at the quadrature points \cite{sagiv2018spline}. We remark that our theory can also be applied to Galerkin-gPC methods \cite{xiu2010numerical}. 

{\bf Density estimation in UQ:} The main appeal of the gPC method is its spectral $L^2$ convergence. As noted above, it is an open question whether this can be used to prove convergence of the PDFs, i.e., an upper bound on $p_{\mu} - p_{\nu}$ in some $L^p$. However,  Theorem \ref{thm:lq_wp} implies that spectral~$L^2$ convergence of $g_N$ to $f$ can yield fast convergence of $W_p (\mu ,\nu)$ for {\em any} $1\leq p < \infty$.

\begin{theorem}\label{thm:gpc}
Let $f$ be analytic in an ellipse in the complex plane that contains $[-1,1]$, and let $d\varrho (\alpha) = k (1-~\alpha)^{\beta _1} (1+~\alpha )^{\beta_2 } d\alpha$, for any $\beta _1 , \beta _2 \in \mathbb{R}$ and a proper normalization constant $k=k(\beta _1 ,\beta _2)$. Let $g(\alpha)$ be the collocation gPC approximation of $f$, i.e., the $N$-th order polynomial interpolant of $f$ at the respective Gauss quadrature points. Then, for every $p\geq 1$, $$W_p (\mu, \nu ) \lesssim e^{-\gamma N} \, , \qquad \|F_{\mu}-F_{\nu} \|_1 \lesssim e^{-\gamma N} \, , \qquad n\to \infty \, ,$$ where $\gamma$ does not depend on $N$.
\end{theorem}

\begin{proof}
If $f$ is analytic, the truncated expansion has the exponential accuracy
$$\| f(\alpha ) - \sum_{n=0}^{N-1 } \hat{f} (n) p_n (\alpha ) \|_2  \lesssim  e^{-\gamma N} \, , \qquad N\gg 1 \, ,$$ for some constant $\gamma >0$ \cite{trefethen2013approximation, wang2012convergence, xiu2010numerical}. Next, since the collocation gPC is a spectrally accurate approximation of the polynomial projection in $L^2$ \cite{ gottlieb1977numerical}, then $\|f-g\|_2 \lesssim e^{-\gamma N}$ as well for $N\gg 1$. Finally, since $\|f-\pi _{N}(f) \|_{\infty}$ does not grow exponentially \cite{ gottlieb1977numerical}, Theorem \ref{thm:lq_wp} applies.
\end{proof}

Two particularly important cases of this theorem are when $\varrho$ is the Lebesgue measure, associated with the Legendre polynomials ($\beta _{1,2} = 0$) and the measure associated with the Chebyshev polynomials ($\beta _{1,2}= -1/2$). By Theorem \ref{thm:gpc}, the convergence of the Wasserstein metric stands in {\em sharp contrast} to that of the PDFs, i.e., of $\|p_{\mu}-p_{\nu}\|_{L^q}$. As previously noted, the convergence of the PDFs for the gPC method has not been proved, and might not be obtained at all for moderate values of $N$ \cite{sagiv2018spline}. It remains an open question whether Theorem \ref{thm:gpc} can be extended to measures with an unbounded support, such as the normal and the exponential distributions. Such a generalization might require a generalization of Theorem \ref{thm:lq_wp} to unbounded domains. We further note that Theorem~\ref{thm:gpc} can be extended to measures $\varrho '$ that are bounded from above by $\varrho$, see \cite{kats2017spectral} for details.
\paragraph{\bf Numerical example.} We approximate the same function $f$, as defined in \eqref{eq:f_numer}, and approximate it using polynomial interpolation at Gauss-Legendre quadrature points (see Sec.\ \ref{sec:gpc_pf}). Since $f$ is analytic, Theorem \ref{thm:gpc} guarantees that the gPC-based $\nu$ converges exponentially in $N$ to that of $\mu$, see Fig.\ \ref{fig:w1_tanh}(b). The convergence of the respective PDFs, on the other hand, is polynomial at best (see Fig.~\ref{fig:w1_tanh}(c)). Quantitatively, the $W_1$ error decreases by $8$ orders of magnitudes between $N=4$ and $N=120$, whereas the $L^1$ between the PDFs decreases by only $4$.

\subsection{Comparison to the histogram method}
This paper, as noted, is motivated by the following class of algorithms: to approximately characterize $\mu = f_* \varrho$, first approximate $f$ by $g$, and then approximate $\mu$ by $\nu =g_* \varrho$. How does this approach compare with more standard statistical methods?
We focus on one common nonparametric statistical density estimation method, the histogram method; Given i.i.d.\ samples from $\mu$, denoted by $f(\alpha _1)=y_1, \ldots , f(\alpha _N) =y_N$, and a partition of the range of~$f(\pmb{\alpha} )$ into $L$ disjoint intervals (bins) $\{ B_{\ell} \}_{\ell=1}^{L}$, the histogram estimator of the PDF is
$$p_{\rm hist}(y) :\,= \frac{1}{N}\sum\limits_{\ell =1}^L  \left( \#~\text{of samples for which}~y_j \in B_{\ell} \right)\cdot \mathbbm{1}_{ B_{\ell} } (y)  \, ,$$
where $\mathbbm{1}_{B_{\ell} }$ is the characteristic function of bin $B_{\ell}$ \cite{wasserman2013all}. The histogram methods is intuitive and easy to implement. What is then the advantage of approximation-based UQ methods? In Sec.~\ref{sec:hist}, using results by Bobkov and Ledoux \cite{bobkov2016one}, we 
prove that 
\begin{corollary}\label{cor:hist}
Under the conditions of Theorem \ref{thm:spline_was}, the $d$-dimensional, $m$-th order spline-based estimator of $\mu$ outperforms the histogram method on average in the $W_p$ sense when $d<2(m+1)$.
\end{corollary}
The average in this corollary refers to all i.i.d.\ realizations of $y_1, \ldots , y_N$ from $\mu$.
This corollary is an example of the so-called ``curse of dimensionality". To maintain a constant resolution and accuracy, the amount of data points (and hence the computational complexity) needs to increase exponentially with the dimension. Hence, above a certain dimension, it is preferable to ignore the underlying structure (i.e., the approximation of $f$ by $g$) and to consider only the empirical distribution of the i.i.d.\ samples $\{ f(\alpha _j) \}_{j=1} ^N$.

\begin{proof}
The error of spline interpolation is controlled by the following theorem
\begin{theorem*}[de Boor \cite{deboor1978splines} and Hall and Meyer \cite{hall1976bounds}]
Let $f\in C^{m+1} \left( \left[ \alpha_{\min} , \alpha_{\max} \right] \right)$, and let $g(\alpha)$ be its "not-a-knot", clamped or natural $m$-th spline interpolant. Then $$
\big\|\big(f(\alpha) -g(\alpha) \big) ^{(j)} \big\|_{L^{\infty} [\alpha_{\min}, \alpha_{\max} ] } \leq C_{\rm spl} ^{(j)}\left\| f^{(m+1)} \right\| _{\infty} h ^{m+1-j} \, , \qquad j=0,1,\ldots , m-1 \, ,$$
where $C_{\rm spl} ^{(j)}>0$ is a universal constant that depends only on the type of boundary condition and $j$, $m$, and $h = \max_{1<j\leq N} \lvert \alpha_j - \alpha _{j-1} \rvert$.
\end{theorem*}
This result is extended for higher dimensions using the the construction of tensor-product grid and tensor-product splines. The definitions here become more technical, and we refer to Schultz~\cite{schultz1969spline} for further detail.
We note that even in the multidimensional case, the error is still bounded by the spacing $h^{m+1-j}$. However, the {\em number} of grid points $N$ is proportional to $h^{-d}$ (this is the so-called curse of dimensionality which we previously mentioned). By the above error bounds, and by Theorem~\ref{thm:wass_gen}, we have that $W_p(\mu , \nu) \leq \|f-g\|_{\infty} \lesssim h^{m+1}$.
\end{proof}

\subsection{Proof of Corollary \ref{cor:hist}}\label{sec:hist}
\begin{proof}
Given $N$ i.i.d.\ from $\mu$, denoted by $y_1,\ldots ,y_N$, define the empirical distribution as $$\mu_{\rm emp} :\,= \frac{1}{N}\sum\limits_{j=1}^{N} \delta_{y_j} \, ,$$ where $\delta _y$ is the Dirac delta distribution centered at the point $y\in \mathbb{R}$. Under certain broad assumptions (see \cite{bobkov2016one} for details), $\mathbb{E}W_p (\mu ,\mu_{\rm emp}) \lesssim N^{-1/2}$, where the expectancy in these bounds is over all realizations of $y_1, \ldots, y_N$ with respect to the measure $\mu$ \cite{bobkov2016one}. 

By the triangle inequality and linearity of expectation, $$\mathbb{E}W_p(\mu, \mu_{\rm hist}) \leq \mathbb{E}\left[W_p (\mu, \mu_{\rm emp}) + W_p (\mu_{\rm emp}, \mu _{\rm hist} )\right] = \mathbb{E}W_p (\mu, \mu_{\rm emp}) + \mathbb{E} W_p (\mu_{\rm emp}, \mu _{\rm hist} ) \, ,$$
where $d\mu_{\rm hist}(y)=p_{\rm hist}(y) \,dy $ is the measure defined by the histogram estimator. It is therefore sufficient to show that $\mathbb{E} W_p ( \mu_{\rm emp} , \mu _{\rm hist}) \lesssim N^{-(1+1/p)}$ for any $ p\geq 1 $. We will prove a slightly stronger claim -- that $W_p ( \mu_{\rm emp} , \mu _{\rm hist}) \lesssim N^{-(1+1/p)}$ for {\em every} set of numbers $y_1, \ldots , y_N$. 

Let $\{B_{\ell}\}_{\ell=1} ^L$ be the bins of the histogram estimator and let $\mu_{{\rm emp},\ell}$ and $\mu_{{\rm hist},\ell}$ be the restriction of the measures $\mu _{\rm emp}$ and $\mu _{\rm hist}$ to $B_{\ell}$, respectively, for every $1\leq \ell \leq L$. By definition, there are exactly~$N\cdot \mu _{{\rm hist}} (B_{\ell})$ samples that fall into $B_{\ell}$, and so $\mu_{{\rm hist},\ell}  (B_{\ell}) = \mu_{{\rm emp},\ell}(B_{\ell})$.
Hence, the two measures $\mu_{{\rm emp},\ell}$ and $\mu_{{\rm hist},\ell}$ are comparable in the Wasserstein metric and we can write that
$$W_p ^p(\mu _{\rm emp} , \mu _{\rm hist}) \leq \sum_{\ell =1}^L W_p ^p(\mu _{{\rm emp},\ell} , \mu _{{\rm hist},\ell}) \, .$$ Since $\mu_{{\rm hist},\ell}$ is uniform on $B_{\ell}$ for any~$\ell$, the Wasserstein distance is the greatest if all of the samples in $ B_{\ell}$ are located on the extreme edge of the bin, i.e., if $y_j \in B_{\ell}$ then $y_j = a_{\ell}$, where we denote $B_{\ell} = [a_{\ell}, b_{\ell}]$. Hence, for every $1\leq \ell \leq L$,
\begin{align*}
W_p ^p(\mu_{{\rm emp},\ell} ,\mu_{{\rm hist},\ell}) &\leq \mu_{{\rm emp},\ell} (B_{\ell}) \int\limits_{a_{\ell}}^{b_{\ell}} (y-a_{\ell})^p \, dy \\
&= \frac{\mu_{{\rm emp},\ell}(B_{\ell})}{p+1} (b_{\ell} - a_{\ell})^{p+1} \, ,
\end{align*}
and so 
\begin{align*}
W_p ^p(\mu_{{\rm emp}},\mu_{{\rm hist}}) &\leq \sum\limits_{\ell=1}^{L} \frac{\mu_{{\rm emp},\ell}(B_{\ell})}{p+1} (b_{\ell} - a_{\ell})^{p+1} \\
&\lesssim N^{-(p+1)}\sum\limits_{\ell=1}^L \mu_{{\rm emp},\ell} (B_{\ell})\\
&= N^{-(p+1)}\sum\limits_{\ell=1}^L \mu_{{\rm emp}} (B_{\ell})=N^{-(p+1)}  \, ,
\end{align*}
where the second inequality is due to the partition, in which $(b_{\ell}-a_{\ell})\sim N^{-1}$, and the last equality holds since $\mu_{\rm emp}$ is a probability measure and since $\{B_{\ell} \}_{\ell =1}^L$ is a partition of its support.
\end{proof}

\section{Acknowledgments}
The author would like to thank S.\ Steinerberger for many useful comments and advice. This research was partially carried out during a stay of the author as a guest of R.R.\ Coifman and the Department of Mathematics at Yale University, whose hospitality is gratefully acknowledged.

\bibliographystyle{unsrt}

\end{document}